\documentclass{amsart}
\usepackage[margin=2.5cm]{geometry}
\usepackage{graphics, amsmath,amssymb, enumitem, comment, url, mathtools}
\usepackage{graphicx}
\usepackage{epsfig}

\newif\ifcolorcomments
\newcommand{\allowcomments}[4]{
\newcommand{#1}[1]{\ifdraft{\ifcolorcomments{\textcolor{#4}{##1 --#3}}\else{\textsl{ ##1 \ --#3}}\fi}\else{}\fi}
}
\usepackage{hyperref}

\colorcommentstrue
\usepackage{amssymb}
\usepackage{amsmath,amsthm}

\usepackage{colortbl}

\newtheorem{theorem}{Theorem}[section]
\newtheorem{lemma}[theorem]{Lemma}
\newtheorem{proposition}[theorem]{Proposition}
\newtheorem{corollary}[theorem]{Corollary}
\theoremstyle{definition}

\newtheorem{remark}[theorem]{Remark}

\newcommand{\DD}{\mathcal D}

\newcommand{\N}{\mathbb N}

\newcommand{\Z}{\mathbb Z}

\newcommand{\mfk}{\mathfrak k}

\newcommand{\rad}{\operatorname{rad}}

\renewcommand{\text}{\textup}

\newcommand{\NPC}[1]{\ignorespaces}

\newif\ifdraft\drafttrue

\def\N{\mathbb N}
\def\Z{\mathbb Z}

\IfFileExists{marvosym.sty}{
\RequirePackage{marvosym} 
}{

}

\IfFileExists{wasysym.sty}{
\RequirePackage{wasysym}
}{

}

\mathtoolsset{showonlyrefs}
\newcommand*{\myDots}{\ifmmode\mathellipsis\else.\kern-0.07em.\kern-0.07em.\fi}

\allowcomments{\comnikita}{NS}{Nikita}{blue}

\newcommand {\ignore}[1] {}

\begin{document}

\title{Radical bound for Zaremba's conjecture}

\author{Nikita Shulga}
\address{Nikita Shulga,  Department of Mathematical and Physical Sciences,  La Trobe University, Bendigo 3552, Australia. }
\email{n.shulga@latrobe.edu.au}
\date{}

\maketitle

\begin{abstract}
Famous Zaremba's conjecture (1971) states that for each positive integer $q\geq2$, there exists positive integer $1\leq a <q$, coprime to $q$, such that if you expand a fraction $a/q$ into a continued fraction $a/q=[a_1,\ldots,a_n]$, all of the coefficients $a_i$'s are bounded by some absolute constant $\mfk$, independent of $q$. Zaremba conjectured that this should hold for $\mfk=5$. In 1986, Niederreiter proved Zaremba's conjecture for numbers of the form $q=2^n,3^n$ with $\mfk=3$ and for $q=5^n$ with $\mfk=4$.
In this paper we prove that for each number $q\neq 2^n,3^n$, there exists $a$, coprime to $q$, such that all of the partial quotients in the continued fraction of $a/q$ are bounded by $ \rad(q)-1$, where $\rad(q)$ is the radical of an integer number, i.e. the product of all distinct prime numbers dividing $q$.

In particular, this means that Zaremba's conjecture holds for numbers $q$ of the form $q=2^n3^m, n,m\in\N \cup \{0\}$ with $\mfk= 5$, generalizing Neiderreiter's result.

Our result also improves upon the recent result by Moshchevitin, Murphy and Shkredov on numbers of the form $q=p^n$, where $p$ is an arbitrary prime and $n$ sufficiently large.
\end{abstract}

\section{Introduction and main result}

Let $q\geq2$ be an integer and let $a$ be an integer with $1 \leq a <q$, such that $\gcd(a,q)=1$. Then a rational number $a/q$ can be expanded into a finite  simple continued fraction as
$$
\frac{a}{q} = \frac{1}{a_1 + \frac{1}{a_2 + \frac{1}{\ddots \, + \frac{1}{a_r}}}} = [a_1,\ldots,a_r], \quad a_i\in \Z_+.
$$
Note that each rational $a/q\in(0,1)$ has two different representations
\begin{equation}\label{twocf}
a/q=[a_1,a_2,\ldots,a_{r-1}, a_r]  
\,\,\,\,\,\text{and}\,\,\,\,\,
a/q=[a_0;a_1,a_2,\ldots,a_{r-1}, a_r-1,1]
,
\,\,\,\, \text{where}\,\,\,\, a_r \ge 2
.
\end{equation}
Denote 
$$
K\left( \frac{a}{q} \right) = \max (a_1,\ldots,a_r).
$$
Zaremba's famous conjecture \cite{MR0343530} states that there is an absolute constant $\mfk$ with the following property: for any positive integer $q$ there exists a coprime to $q$, such that $K(a/q)\leq \mfk$. In fact, Zaremba's conjectured that for $\mfk=5$. For large prime numbers $q$ Hensley (see \cite{MR1306033}, \cite{MR1387719}) conjectured that even $\mfk=2$ should be enough. Note that this cannot be improved to $\mfk=1$, as all finite continued fractions with all partial quotients equals $1$, have Fibonacci numbers as denominators.

Korobov \cite{MR0157483} showed that for prime $q$ there exists $1 \leq a <q$ such that
$$
K(a/q) \ll \log q.
$$
In fact, this result is also true for composite $q$.

In 1986, Neiderreiter \cite{MR0851952} proved Zaremba's conjecture for the special cases of $q=2^n,\,3^n,\,5^n$ for any $n \geq 1$ with a better bound of $K(a/q)\leq3$ for powers of $2$ and $3$ and with $K(a/q)\leq4$ for powers of $5$. Yodphotong and Laohakosol \cite{MR2010939} proved this conjecture for the case $q=6^n$ for any $n\in\N$ with a bound $K(a/q)\leq5$. Finally, Komatsu \cite{MR2138391} proved it for $q=7^{c2^n}, (n\geq0, c=1,3,5,7,9,11)$ with a bound $K(a/m)\leq3$.

In recent years, Zaremba's conjecture was extensively studied, see \cite{MR3194813}, \cite{MR3284129}, \cite{MR3361774}, \cite{MR4460483},  \cite{MR4145818}, \cite{MR4202008} and many others. One of the most recent result is due to Moshchevitin, Murphy and Shkredov \cite{moshchevitin2022korobov}, they proved the following theorem.
\begin{theorem}\label{shkredov}
Let $q$ be a positive sufficiently large integer with sufficiently large prime factors.
Then there is a positive integer $a$ with $\gcd(a, q) = 1$ and
\begin{equation}\label{shk1}
    M = O(\log q/ \log \log q) 
\end{equation}
such that
\begin{equation}\label{shk2}
\frac{a}{q} = [ c_1,\ldots, c_s], \quad\quad c_j \leq M , \quad\quad \forall j \in\{1,\ldots,s\}.
\end{equation}
Also, if $q$ is a sufficiently large square–free number, then \eqref{shk1}, \eqref{shk2} take place.

\noindent Finally, if $q = p^n$, $p$ is an arbitrary prime, then \eqref{shk1}, \eqref{shk2} hold for sufficiently large $n$.
\end{theorem}

Zaremba's conjecture was also studied in other settings. See, for example, \cite{Malagoli} for overview of the results on polynomial analogue of Zaremba's conjecture and \cite{GGRMHSH}, where Hurwitz continued fraction analogue of Zaremba's conjecture was formulated and some partial results are given.

Recall the definition of a radical of an integer number $n$. It is equal to the product of the distinct prime numbers dividing the given number. Formally, 
$$
\rad(n) = \prod\limits_{\substack{ p|n \\ p \text{ prime } }} p.
$$
In this paper we prove the following result in direction of Zaremba's conjecture.
\begin{theorem}\label{thmradical}
For any integer $q\geq2$, such that $q\neq2^n,3^n$,  there exists a positive integer $a$ with $1\leq a <q$ and $\gcd(a,q)=1$, such that $K(a/q)\leq \rad (q)-1$.
\end{theorem}

\begin{remark}
Note that cases $q=2^n$ and $q=3^n$ are not covered by Theorem \ref{thmradical}. As Niederreiter showed in his paper, bound $\mfk=3$ is optimal for both $q=2^n$ and $q=3^n$. In terms of a radical, this means that for $q=2^n$ there exists an odd integer $a$ with $K(a/q)\leq \rad(q)+1$ and for $q=3^n$, there exists a positive integer $a$ with $\gcd(a,3)$ with $K(a/q)\leq \rad(q)$, both of which are slightly worse than the statement of Theorem \ref{thmradical}.
\end{remark}

Theorem \ref{thmradical} provides a generalization of results by Neiderreiter and Yodphotong-Laohakosol for a set of numbers of the form $q=2^n3^m$, where $n,m\in\N$.
\begin{corollary}\label{2n3m}
For any non-negative integers $n,m$ there exists a positive integer $a$ with $1\leq a < 2^n3^m$ and $\gcd(a,2^n3^m)=1$, such that $K(a/2^n3^m)\leq 5$.
\end{corollary}


We also note that we can improve one of the statements of Theorem \ref{shkredov}.
\begin{remark}
Using Theorem \ref{thmradical} we can improve Theorem \ref{shkredov} in the case $q=p^n$ for a prime number $p$. By Theorem \ref{shkredov} for $n$ sufficiently large one has 
$$
K(a/q) \leq O ( n \log p/ \log (n\log p) )
$$
for some $1 \leq a < q$, coprime to $p$. For large enough $n$, say for $n\asymp p^2$, our Theorem \ref{thmradical} gives a better bound of $K(a/q)\leq p-1$, as compared to 
$$
K(a/q) \leq O ( p^2 \log p/ \log (p^2\log p) ) = O(p^2)
$$
from Theorem \ref{shkredov}. When $n\gg p^2$, the bound obtained by Theorem \ref{thmradical} remains the same, but the bound obtained by Theorem \ref{shkredov} will become worse the larger the value of $n$ is.
\end{remark}
In the end of Section \ref{sec:proof} we combine our approach with Theorem \ref{shkredov} to get a generalization of it in the case of $q$ being a sufficiently large square-free integer, see Remark \ref{shkredov:generalization} below.


\section{Proof of the main result}\label{sec:proof}
The main construction in the proof rely on a famous folklore statement, knows as Folding lemma. For the reference, see \cite{MR1149740}.
\begin{lemma}[Folding Lemma]\label{folding}
If $t_r/q_r=[a_1,\ldots,a_r]$ and $b$ is a non-negative integer, then
\begin{equation}\label{foldingeq}
\frac{t_r}{q_r} + \frac{(-1)^r}{bq_r^2} = [a_1,\ldots,a_r,b-1,1,a_r-1,a_{r-1},\ldots,a_1].
\end{equation}
\end{lemma}

\begin{remark}
   We note that Lemma \ref{folding} and the proof of Neiderreiter's result on Zaremba's conjecture for powers of $2$ can be reformulated in terms of Minkowski question mark function $?(x)$. For more details, see  \cite{MR4121875} (Lemma 3.2 and Remark 1).
\end{remark} 

We combine some trivial observations, which will be used in the proof of Theorem \ref{thmradical} in the following remark.

\begin{remark}\label{trivial}
The following assertions hold.
\begin{enumerate}
\item\label{trivial1} $\gcd(t_r q_r b +(-1)^n,bq_r^2)=1$, meaning that the right-hand side of \eqref{foldingeq} is a reduced fraction with denominator $bq_r^2$.

\item\label{trivial2} We can allow intermediate partial quotients to be equal to $0$ using the convention $[\ldots,x,0,y,\ldots]= [\ldots,x+y,\ldots]$.

\item\label{trivial3} As each fraction $t_r/q_r$ has two different continued fraction expansions of different length parity, we can always choose $r$, such that right-hand side of \eqref{foldingeq} will look like $[a_1,\ldots,a_r,b-1,1,a_r-1,a_{r-1},\ldots,a_1]$ and not like
$[a_1,\ldots,a_{r-1},a_r-1,1,b-1,a_r,\ldots,a_1]$ with $a_r\geq2$.
\end{enumerate}

\end{remark}
Throughout the proof, when applying Lemma \ref{folding}, we will be using considerations from the statement  \eqref{trivial3} from Remark \ref{trivial}, meaning that we are assuming that $r$ is of needed parity for the resulting fraction on the right-hand side of \eqref{foldingeq} to be of the form $[a_1,\ldots,a_r,b-1,1,a_r-1,a_{r-1},\ldots,a_1]$ with $a_r\geq2$.

Now we are ready to prove Theorem \ref{thmradical}.
\begin{proof}
The canonical representation of the number $q$ is
\begin{equation}\label{canonical}
q=p_1^{n_1}p_2^{n_2}\cdots p_k^{n_k},
\end{equation}
where $p_1<p_2<\ldots<p_k$ are primes are $n_i$ are positive integers. Hence $\rad(q)= p_1\cdots p_k$.

Now consider a following iterative procedure: set $q_{(0)}:= q$ and for $i\geq1$ define $q_{(i)}$ from the equality
\begin{equation}\label{iterative}
q_{(i-1)} = p_{(i)}\cdot q_{(i)}^2,
\end{equation}
where

$$
 p_{(i)}=p_1^{v_1^{(i)}}p_2^{v_2^{(i)}}\cdots p_k^{v_k^{(i)}}  \text{ with } v_j^{(i)}\in\{0,1\} \text{ for all } i,j.$$ Note that for all $i$ we have $p_{(i)} | \rad(q)$ , so, in particular, $p_{(i)}\leq \rad(q)$.
 
Note that by the definition of the procedure \eqref{iterative}, there exists $N\in\N\cup\{0\}$, such that $q_{(N)}>1$ and $q_{(N+1)}=1$. After this step, the process terminates and we have $q_{(N+j)}=p_{(N+j)}=1,\, j\geq2$. Easy to see that $N=0$ if and only if in \eqref{canonical} one has $n_1=\ldots=n_k=1$ and $N\geq1$ otherwise. We have $2\leq q_{(N)} \leq \rad (q)$. We also note that $q_{(N)}$ is a square-free number. 

If $N=0$, then $\rad (q) = q$. Consider a fraction $(q-1)/q$. It has continued fraction expansion $(q-1)/q = [1,q-2,1]$. Hence for $a=q-1$ we get $K(a/q)\leq \rad (q)-2$.

If $N\geq1$, we distinguish several cases.

\textbf{Case 1.} $q_{(N)}\neq2,3,6$. As $q_{(N)}$ is a square-free number, this means that $q_{(N)}\geq5$, and, in particular, $\rad (q)\geq q_{(N)}\geq 5$. For a square-free number $q_{(N)}\geq5$ we have $\varphi(q_{(N)})\geq4$, where $\varphi(n)$ is an Euler's totient function. 

We will make use of the following simple observation.
\begin{proposition}\label{existcf}
For an integer number $q$ with $\varphi(q)\geq4$, there exists an integer $1\leq a \leq q-1$, coprime to $q$, such that 
$$
\frac{a}{q}= [a_1,\ldots,a_n]
$$
with the following properties:
\begin{itemize}
    \item $n\geq2$;
    \item $a_1\geq2$ and $a_n\geq2$;
    \item $K(a/q) \leq \frac{q-1}{2}$.  
\end{itemize}
\end{proposition}
\begin{proof}
First, the only reduced fraction with denominator $q$ with continued fraction of the length $n=1$ is $1/q$. So there are at least three reduced fractions with $n\geq2$. 

Next, note that if $\gcd(a,q)=1$, then also $\gcd(q-a,q)=1$. So amongst those at least three reduced fractions of the length $n\geq2$, we can always find the one with $a_1\geq2$, because if $a/q =[1,\ldots]$, then $(q-a)/q=[a_1,\ldots]$ with $a_1\geq2.$ We can also always guarantee $a_n\geq2$ due to \eqref{twocf}.

The last point is just a simple corollary of the first two. Indeed, take $1\leq a \leq q-1$ with properties guaranteed by the first two points of the proposition. Assume that there exists $1\leq i\leq n$, such that $a_i>(q-1)/2.$

Denote by $\langle a_1,\ldots,a_n\rangle$ the denominator of continued fraction $[a_1,\ldots,a_n].$ Note that 
$$
\langle a_1,\ldots,a_n\rangle = \langle a_n,\ldots,a_1\rangle.
$$

Then for $a/q$ and $j\in\{1,n\}$ one has 
$$q> \langle a_j,\ldots, (q-1)/2,\ldots \rangle \geq \langle 2, (q-1)/2 \rangle = 2 \frac{q-1}{2}+1 = q,
    $$
    which is a contradiction. Hence $K(a/q) \leq \frac{q-1}{2}$.
\end{proof}

We will now apply Proposition \ref{existcf} with $q=q_{(N)}$ to get a reduced fraction $a/q_{(N)}$, satisfying properties from Proposition \ref{existcf}, i.e $a/q=[a_1,\ldots,a_n]$ with $n\geq2,\, a_1,a_n\geq2$ and $K(a/q_{(N)}) \leq (q_{(N)} - 1)/2$.


Application of Lemma \ref{folding} with 
$$ \frac{t_r}{q_r} =  \frac{a}{q_{(N)}} \quad \text{ and } \quad b=p_{(N)}$$
leads to 
$$
\frac{t_r}{q_r} + \frac{(-1)^r}{bq_r^2} = [a_1,\ldots,a_n,p_{(N)}-1,1,a_n-1,\ldots,a_1]= \frac{a_{(N-1)}}{p_{(N)}\cdot q_{(N)}^2} = \frac{a_{(N-1)}}{ q_{(N-1)}}.
$$
We have two possible situations.

If $p_{(N)}=1$, then by statement \eqref{trivial2} from Remark \ref{trivial} we get 
$$
K\left(\frac{a_{(N-1)}}{ q_{(N-1)}}\right)  \leq (q_{(N)}-1)/2+1 =(q_{(N)}+1)/2 \leq \rad (q)-1.
$$
To get the last inequality we used $q_{(N)}\leq \rad(q)$ and $\rad(q)\geq5$.

If $p_{(N)}\geq2$, then 
$$
K\left(\frac{a_{(N-1)}}{ q_{(N-1)}}\right)  \leq \max \left(  p_{(N)}-1, K\left(\frac{a}{q_{(N)}} \right) \right) \leq \max \left( p_{(N)}-1, \frac{q_{(N)} - 1}{2} \right) \leq\rad(q)-1.
$$
Now we iteratively apply Lemma \ref{folding} for $i=1,2,\ldots,N- 1 $ with
$$ 
\frac{t_r}{q_r} =  \frac{a_{(N-i)}}{q_{(N-i)}} \quad \text{ and } \quad b=p_{(N-i)}
$$
to get
$$
\frac{t_r}{q_r} + \frac{(-1)^r}{bq_r^2} = [a_1,\ldots,a_1,p_{(N-i)}-1,1,a_1-1,\ldots,a_1]= \frac{a_{(N-i-1)}}{p_{(N-i)}\cdot q_{(N-i)}^2} = \frac{a_{(N-i-1)}}{ q_{(N-i-1)}}.
$$
If $p_{(N-i)}=1$, then by statement \eqref{trivial2} from Remark \ref{trivial} we get 
$$
K\left(\frac{a_{(N-i)}}{ q_{(N-i)}}\right)  =  \max\left( \frac{q_{(N)}-1}{2}+1, K\left(\frac{a_{(N-i+1)}}{ q_{(N-i+1)}}\right) \right) \leq \rad(q)-1.
$$
If $p_{(N)}\geq2$, then 
$$
K\left(\frac{a_{(N-1)}}{ q_{(N-1)}}\right)  \leq \max \left( p_{(N-i)}-1,  K\left(\frac{a_{(N-i+1)}}{ q_{(N-i+1)}}\right)  \right)   \leq \rad (q)-1 .
$$
Af the step $i=N-1$ after the application of Lemma \ref{folding} we get a fraction with denominator $q_{(0)}=q$ with all partial quotients bounded by $\rad(q)-1$ and the process terminates.

\textbf{Case 2.} $q_{(N)}=2,3$ or $6$. By the assumptions of Theorem \ref{thmradical} we know that $q\neq2^n, 3^n$. Assume further that $q\neq2^n3^m,n,m\geq1$. Then in all cases of $q_{(N)}=2,3$ or $6$, we have $\rad(q)\geq 2\cdot5 = 10.$ For any possible value of $q_{(N)}$ in this case, consider a fraction $1/q_{(N)} = [q_{(N)}].$ We split into two subcases.
\begin{enumerate}
\item For $q_{(N)}=2$, the application of Lemma \ref{folding} for $t_r/q_r =1/2$ and $b=p_{(N)}$ gives us 
$$
\frac{t_r}{q_r} + \frac{(-1)^r}{bq_r^2} = [q_{(N)},p_{(N)}-1,1,q_{(N)}-1]=[2,p_{(N)}-1,2]= \frac{a_{(N-1)}}{p_{(N)}\cdot q_{(N)}^2} = \frac{a_{(N-1)}}{ q_{(N-1)}}.
$$

If $p_{(N)}\geq2$, then the resulting fraction $[2,p_{(N)}-1,2]$ satisfies the first two properties from Proposition \ref{existcf} and $K(a_{(N-1)}/ q_{(N-1)}) \leq \rad(q) -1$, so we can continue the procedure as in Case 1.

If $p_{(N)}=1$, then the resulting continued fraction is $a_{(N-1)}/ q_{(N-1)}=1/4=[4]$.
Next application of Lemma \ref{folding} for $t_r/q_r =1/4$ and $b=p_{(N-1)}$ yields
$$
\frac{t_r}{q_r} + \frac{(-1)^r}{bq_r^2} = [4,p_{(N-1)}-1,1,3]= \frac{a_{(N-2)}}{p_{(N-1)}\cdot q_{(N-1)}^2} = \frac{a_{(N-2)}}{ q_{(N-2)}}.
$$
The resulting continued fraction $[4,p_{(N-1)}-1,1,3]$ satisfy the first two properties from Proposition \ref{existcf} and $K(a_{(N-2)}/ q_{(N-2)}) \leq \rad(q) -1$, so we can continue the procedure as in Case 1.

\item For $q_{(N)}=3$ or $6$, the application of Lemma \ref{folding} for $t_r/q_r =1/q_{(N)}$ and $b=p_{(N)}$ gives us 
$$
\frac{t_r}{q_r} + \frac{(-1)^r}{bq_r^2} = [q_{(N)},p_{(N)}-1,1,q_{(N)}-1]= \frac{a_{(N-1)}}{p_{(N)}\cdot q_{(N)}^2} = \frac{a_{(N-1)}}{ q_{(N-1)}}.
$$
The resulting continued fraction $[q_{(N)},p_{(N)}-1,1,q_{(N)}-1]$ satisfy the first two properties from Proposition \ref{existcf} and $K(a_{(N-1)}/ q_{(N-1)}) \leq \rad(q) -1$, so we can continue the procedure as in Case 1.

So under the assumption $q\neq2^n3^m,n,m\geq1$ we proved the statement of the theorem. Final case is $q=2^n3^m,n,m\geq1$. As $\rad(q)=6$, we want to prove the existence of $1\leq a \leq q-1$ with $\gcd(a,q)=1$, such that $K(a/q)\leq 5$.

We know that the potential values of $q_{(N)}$ are $q_{(N)}=2,3$ or $6$. 
Fractions $1/2=[2],1/3=[2,1],1/6=[5,1]$ satisfy $K(1/q_{(N)})\leq5$. By definition, $q_{(N-1)}=  p_{(N)} q_{(N)}^2$, where $p_{(N)} \in \{1,2,3,6\}$ and $q_{(N)}\in\{2,3,6\}$.
Consider the fractions
\begin{align*}
&\frac{1}{2^2}=[4], \quad &&\frac{3}{2\cdot2^2}=[2,1,2], \quad&& \frac{5}{3\cdot2^2} = [2,2,2],\quad&& \frac{7}{6\cdot2^2}=[3,2,3],\\
&\frac{2}{3^2} = [4,2],\quad &&\frac{5}{2\cdot3^2}=[3,1,1,2], \quad&& \frac{8}{3\cdot3^2}=[3,2,1,2],\quad&& \frac{17}{6\cdot3^2}=[3,5,1,2],\\
&\frac{11}{6^2}=[3,3,1,2],\quad&& \frac{17}{2\cdot6^2}=[4,4,4], \quad&&\frac{23}{3\cdot6^2}=[4,1,2,3,2],\quad&& \frac{49}{6\cdot6^2}=[4,2,2,4,2].
\end{align*}
These fraction cover all possible values of $q_{(N-1)}$ and for all fractions $a/q_{(N-1)}$, one has $K(a/q_{(N-1)})\leq5$. For every fraction except $1/2^2=[4]$, the length of continued fraction expansion is $n\geq2$ and $2\leq a_1,a_n\leq4$. We can apply Lemma \ref{folding} to each of those fractions, hence covering all possible values of $q_{(N-2)}$ (except the ones, generated by $q_{(N-1)}=2^2$, we will deal with them later). Note that in both possible forms of Folding lemma applied to the continued fraction $[a_1,\ldots,a_n]$, i.e.
$$
[a_1,\ldots,a_n,b-1,1,a_n-1,\ldots,a_1], \quad \text{when $b\geq2$}
$$
and
$$
[a_1,\ldots ,a_{n-1},a_n+1,a_n-1,a_{n-1},\ldots,a_1], \quad \text{when $b=1$},
$$
we will not get partial quotients larger than $5$, as we are either adding a partial quotient $b-1$, which in our case is at most $5$, or we increase the last partial quotient $a_n$ by $1$ and put this partial quotient in the middle of our continued fraction. As $a_n\leq4$, we have $a_n+1\leq5$. 

Now when $a/q_{(N-1)}=1/4$, consider all possible values of $q_{(N-2)}$, generated by it:
\begin{align*}
\frac{7}{2^4}=[2,3,2], \quad \frac{9}{2\cdot2^4}=[3,1,1,4], \quad \frac{13}{3\cdot2^4}=[3,1,2,4],\quad \frac{29}{6\cdot2^4}=[3,3,4,2].
\end{align*}
As before, for each continued fraction here, its length $n$ satisfies $n\geq2$ and $2\leq a_1,a_n\leq4$. Iterative application of Proposition \ref{folding} will not create partial quotients larger than $5$ by the same argument as in the last paragraph.

Hence, we have covered all possible values of $q_{(N)},q_{(N-1)}$ and $q_{(N-2)}$, preserving the property $K(a/q)\leq5$. Continuing the procedure until we get a fraction with denominator $q_{(0)}=q$ terminates the process and finishes the proof.
\end{enumerate}

\end{proof}

\begin{remark}
    Note that the constant $\mfk=5$ for numbers $q$ of the form $q=2^n3^m,n,m\geq1$ is optimal, see, for example, $q=2\cdot3$ or $q=2\cdot 3^3$.
\end{remark}
Finally, we provide a remark on how one can apply our framework combined with results such as Theorem \ref{shkredov} to get some  generalizations of them.
\begin{remark}\label{shkredov:generalization}
By Theorem \ref{shkredov}, when $d$ is a sufficiently large square-free integer, there exists $a$, coprime to $d$, such that 
$$
K(a/d) \leq O (  \log d/ \log \log d ).
$$
By the iterative procedure \eqref{iterative} from the proof of Theorem \ref{thmradical}, we reduce any number $q$ to a square-free number $q_{(N)}$ in finite number of steps. Afterwards, we start building continued fractions with given properties starting from $q_{(N)}$ and going back to the number $q:=q_{(0)}$.

    Theorem \ref{shkredov} guarantees that for the square-free number $d=q_{(N)}=p_1\cdots p_k$ one can find an integer $a$, coprime to $q_{(N)}$ with
$$
K(a/q_{(N)}) \leq O (  \log q_{(N)}/ \log \log q_{(N)} ).
$$
Starting from this fraction as a seed fraction of our iterative procedure, we can apply Lemma \ref{folding} with $b\in\DD$, where $\DD$ is the set of small divisors of $q_{(N)}$ defined as
$$
\DD = \left\{ m\in\N\, : \, m | q_{(N)} \, \text{ and } \, m \leq O (  \log q_{(N)}/ \log \log q_{(N)} )\right\}.
$$
This set is always non-empty, because $1\in\DD$.
Then using the same iterative procedure as in the proof of Theorem \ref{thmradical}, one will generate all numbers $q$ of the form 
$$
p_{i_1}^{n_1}\cdots p_{i_j}^{n_j} d^{2^n} \quad \text{ for all } \quad p_{i_1}\cdots p_{i_j}\in \DD,\, n_1,\ldots,n_j\in\N\cup\{0\}, n\in\N.
$$

In particular, as $1\in\DD$, iterative application of Lemma \ref{folding} will generate all numbers of the form $d^{2^n},n\in\N$ with the same bound
$$
K(a/d^{2^n}) \leq O (  \log d/ \log \log d )
$$
for sufficiently large square-free integer $d$ and some number $a$ coprime to $d$.
\end{remark}


\bibliographystyle{abbrv}

\end{document}